\documentclass[11pt,reqno]{amsart}
\usepackage[body={7in,9in},centering]{geometry}
\usepackage{amsmath,amssymb,amsthm,mathrsfs}
\usepackage{color,xcolor}
\definecolor{cobalt}{RGB}{61,99,181}
\usepackage[colorlinks,citecolor=cobalt,linkcolor=cobalt]{hyperref}
\usepackage{float}
\usepackage{exscale}
\usepackage{relsize}
\usepackage{graphicx}
\usepackage{tikz}
\newtheorem{thm}{Theorem}[section]
\newtheorem{defi}[thm]{Definition}

\newtheorem{lem}[thm]{Lemma}

\newtheorem{prop}[thm]{Proposition}

\numberwithin{equation}{section}

\date{\today}

\makeatletter

\newcommand{\Rmnum}[1]{\expandafter\@slowromancap\romannumeral #1@}
\makeatother

\begin{document}

\title[\(C^{*}(G,A)\)]{The inverse-closed subalgebra of \(C^{*}(G,A)\)}

\author[JianJun Chen]{JianJun Chen}
\address{
\textsuperscript{3}
College of Mathematics and Statistics, Chongqing University, Chongqing, 401331, P. R. China}
\email{202206021014@stu.cqu.edu.cn}


\begin{abstract}
 This paper studies the inverse-closed subalgebras of the Roe algebra with coefficients of the type \(l^2(G, A)\). The coefficient \(A\) is chosen to be a non-commutative \(C^*\)-algebra, and the object of study is \(C^*(G, A)\) generated by the countable discrete group \(G\). By referring to the Sobolev-type algebra, the intersection of a family of Banach algebras is taken. It is proved that the intersection \(W_a^{\infty}(G, A)\) of Banach spaces is a spectrally invariant dense subalgebra of \(C^*(G, A)\), and a sufficient condition for this is that the group action of \(G\) has polynomial growth.
\end{abstract} \maketitle

\section{Introduction}\label{Intro}

The Roe algebra is an indispensable object of study in cutting-edge research topics such as the theory of coarse embeddings, the index theorem for non-compact manifolds, and the classification of high-dimensional topological manifolds. 

As a very special class of subalgebras within the Roe algebra, inverse-closed algebras possess an inverse structure that is exactly the same as that of the parent algebra. This characteristic gives them unique advantages in related research. The \(K\)-groups of an inverse-closed subalgebra and the original \(C^*\)-algebra are isomorphic, that is, \(K_i(\mathcal{A})\cong K_i(\mathcal{B}),\quad i = 0,1\). This is because the property of being an inverse-closed dense subalgebra ensures the consistency of \(K\)-theory \cite{oyono2015quantitative}. Therefore, in many cases, a subalgebra with a simpler structure can be selected, making the process of computing the \(K\)-group easier. Moreover, spectrally invariant subalgebras can also preserve some properties in \(K\)-theory. Hence, it is of great importance to study the inverse-closed property of the Roe algebra with coefficients.

In Chapter 3, the inverse-closed subalgebras of \(C^*(G,A)\) are discussed. To study the inverse-closed subalgebras of \(G^*(G,A)\) with coefficient \(A\) and generated by the group \(G\), first, when the group action has polynomial growth, polynomial weighting is taken. The intersection of the family of Banach algebras \(\{W_a(G,A)\}\) is taken to construct \(W_a^{\infty}(G,A)\). It is proved that this subalgebra \(W_a^{\infty}(G,A)\) is an inverse-closed subalgebra of \(C^*(G,A)\).

\section{Preliminary}\label{S2}
Let \(G\) be a countable discrete group. Then there must exist a left-invariant metric on \(G\) such that \(G\) is a metric space with bounded geometry. That is, \(d_l(\gamma_1,\gamma_2) = l(\gamma_2^{-1}\gamma_1)\). Then \(d_{l}\) is the metric on \(G\) induced by \(l\), which means that \(d_l\) is left-invariant, and for all \(\gamma_1,\gamma_2,t\in G\), \(d_l(t\gamma_1,t\gamma_2) = d_l(\gamma_1,\gamma_2)\).

\begin{defi}
	As previously described, \((G, d_{l})\) is given. Here, \(\mathbb{A}\) is a non - commutative \(C^*\) algebra with a unit element. The pre - Roe algebra of \(G\) with coefficients is defined as follows:
	\[C[G,A]=\left \{ \phi:G\times G\to \mathbb{A}\mid \phi\text{ is a locally compact operator and has finite propagation}   \right \}.\]
	\begin{defi}
	The closure of \(C[G,A]\) under the operator norm is denoted as \(C^{*}(G,A)\), which is the Roe algebra with coefficients generated by the countable discrete group \(G\). That is
	\[C^{*}(G,A)=\overline{C[G,A]}^{\left \| \cdot  \right \|_{\mathcal{B}(\ell ^{2}(G,A))}}.\]
	\end{defi}
	For any \(\phi\in C^*(G,A)\), we equip  $\phi\ $with the norm . 
	\[\|T\|_{2}=\|T\|_{B(l^2(G,A))}\]
	And it is a bounded operator,in other words, it can be regarded as a bounded operator on \(\ell^{2}(G,A)\) through convolution. That is, \(\phi:\ell ^{2}(G,A)\longrightarrow \ell^{2}(G,A)\) and
	\[\phi\ast \eta (\gamma_1)=\sum _{\gamma_2\in G}\|\varphi(\gamma_1,\gamma_2)\|\eta(\gamma_2),\forall \gamma_1\in G.\]	
\end{defi} 
\begin{defi}
Define the truncation operator of 
$T _{n}=\{t_{n}(\gamma_1,\gamma_2)\}	$ :
	\begin{equation*}
		t_{n}(\gamma_1,\gamma_2)=\left\{
		\begin{aligned}
			t(\gamma_1,\gamma_2) & , &  d_{l} (\gamma_1,\gamma_2)\le n.\\
			0 & , & \text{others.}
		\end{aligned}
		\right.
	\end{equation*}
\end{defi}

\section{The spectra of Toeplitz operators}\label{spectra}
First, we present a Banach subalgebra.
\[ W_{a}(G,A)= \left \{ T:G\times G\longrightarrow \mathbb{A}\mid  \sqrt{\sum_{\gamma\in G} \Big ( \sup_{\left \{ \gamma\in G \right \} }\| t(\gamma )\|  \Big ) ^2 \left ( u(d(\gamma) \right )^{2}}<\infty  \right  \} . \]

\begin{thm} Define $W_{a}^{\infty }(G,A)$ :
\begin{align*}
	\left \| T \right \|_{a} &=  \sqrt{\sum_{\gamma\in G} \Big ( \sup_{\left \{ \gamma_1,\gamma_2\in G,\gamma_2^{-1}\gamma_1=\gamma\right \} }\|t\left ( \gamma_1,\gamma_2 \right )  \|   \Big ) ^2 \left ( 1+l(\gamma_2^{-1}\gamma_1) \right )^{2a}} \\
	&=\sqrt{\sum_{\gamma\in G} \Big ( \sup_{\left \{ \gamma\in G \right \} }\| t(\gamma )\|  \Big ) ^2 \left ( 1+l(\gamma) \right )^{2a}}
	=\|\tilde tu\|_2 .	
\end{align*}
\[u(\gamma_1,\gamma_2)=(1+d_{l}(\gamma_1,\gamma_2))^a.\quad a>0.\]
\[W_{a}^{\infty }(G,A)=\bigcap _{a=0}^{\infty}W_a(G,A).\]	
\end{thm}
\begin{defi}
$G$ is said to have $u$-growth if $\#B(e,r)<cu(r)^t$, for some $c,t>0$.
and $G$ is said to have $(u,2)$-decay if $W_{a}^{\infty}(G,A) \subset C^*(G,A)$.
\end{defi}
\begin{prop}
When \(G\) has \(u\)-growth, \((G, d_l)\) is \((u, 2)\)-decaying, and the operator can be approximated in norm by the truncation operator.
\end{prop}
\begin{proof}
	We only need make $m-2a<-2$. 
	\begin{align*}
	\left \| T \xi  \right \|_2^{2} & =\Big( \sum _{\gamma_1\in G} \Big| \sum_{\gamma_2\in G}||t (\gamma_1,\gamma_2) ||(u(\gamma_1,\gamma_2))(u(\gamma_1,\gamma_2))^{-1}\Big |\xi (\gamma_2) \Big|^{2}\Big)\\
	& \le \sum _{\gamma_1\in G}\Big( \Big| \sum_{\gamma_2\in G}||t (\gamma_1,\gamma_2) ||(u(\gamma_1,\gamma_2))\Big |^2\Big)(\sum _{\gamma_2\in G}(u(\gamma_1,\gamma_2))^{-1}\Big |\xi (\gamma_2) \Big|)^{2}\\
	&= \sum _{\gamma_1\in G} \sum_{\gamma_2\in G}\left \|t (\gamma_1,\gamma_2)\ | \right | ^{2}(1+l(\gamma_2^{-1}\gamma_1))^{2a} \sum_{\gamma_2\in G}(1+l(\gamma_2^{-1}\gamma_1))^{-2a}\left | \xi (\gamma_2) \right | ^{2}\\
	&\le \|T\|_a^2\sum_{\gamma_1} \sum_{\gamma_2} (1+l(\gamma_2^{-1}\gamma_1))^{-2a}\left | \xi (\gamma_2) \right | ^{2}.\\
	&\le  \|T\|_a^2 \sum_{\gamma_2} \left | \xi (\gamma_2) \right | ^{2} \sum_{\gamma_1}\sum _{\gamma_2}(1+l(\gamma_2^{-1}\gamma_1))^{-2a}\\
	&= \|T\|_a^2 \left \| \xi  \right \|_2^{2} \sum_{\gamma_2}\sum_{\gamma_1} (1+l(\gamma_2^{-1}\gamma_1))^{-2a}\\
	&\le \|T\|_a^2\left \| \xi  \right \|_2^{2}\sum_{n=0}^{\infty\gamma_2 } \sum _{n\le l(\gamma)<n+1}(1+l(\gamma))^{-2a}\\
	&\le \|T\|_a^2 \left \| \xi  \right \|_2^{2}\sum_{n=0}^{\infty } \left( \left | B(\gamma,n+1) \right |   \right)  (1+n)^{-2a}\\
	&\le \|T\|_a^2 \left \| \xi  \right \|_2^{2}\sum_{n=0}^{\infty } \left( 1+n   \right)^{d(G,d_l)-2a}\le \|T\|_a^2\left \| \xi  \right \|_2^{2}\sum_{n=0}^{\infty } \left( 1+n   \right)^{m-2a}\\
	&\le \|T\|_a^2\left \| \xi  \right \|_2^{2}\sum_{n=0}^{\infty }(1+n)^{-2} \le C_0^2 \|T\|_a^2 \left \| \xi  \right \|_2^{2}.
\end{align*}
then,$\left \| T \right \|_{2}\le\left \| T \right \|_{a}$.
	\begin{align*}
	\left \| T -T_{n} \right \| _{a}^{2}&\le C\left \| T  \right \| _{a+r}^{2}\sum_{n=m+1}^{\infty } (1+n)^{m-2}\\
	&\le C\left \| T  \right \| _{a+r}^{2}\sum_{n=m+1}^{\infty } (1+n)^{-2}\\
	&<\varepsilon.    
\end{align*}
\end{proof}
\begin{lem}
$\|T^n\|_a\le C\|T\|_2^\alpha\cdot\|T\|_a^\beta$.	
\end{lem}
\begin{thm}\label{spectral} 
	\(W_{a}^{\infty}(G,A)\) is a dense inverse-closed Banach subalgebra of \(C^{*}(G,A)\).
\end{thm}
\begin{proof}
	It suffices to prove its inverse-closed property. Let \(Inv~ W^\infty(G,A)\) be the set of all invertible elements of \(W^\infty(G,A)\). For any \(T\in W_{a}^{\infty}(G,A)\), then \(T^{-1}\in C^*(G,A)\). If \(T^{-1}\in Inv~ W^\infty(G,A)\), thus \(W_{a}^{\infty}(G,A)\) is inverse-closed.
	\[Inv W^\infty(G,A)=\{T:T,T^{-1}\in W_{a}^{\infty}(G,A),\text{ and }T\cdot T^{-1}=T^{-1}\cdot T=I\}\]
	For any \(T\in W_{a}^{\infty}(G,A)\), the corresponding adjoint matrix is \(T^{*}=(t(\gamma_1,\gamma_2)^* )_{\gamma_1,\gamma_2\in G}\). Due to the definition of the operator, we have \(T^{-1}=(T^{*}T)^{-1}T^{*} \in C^*(G,A)\). And we have
	\[\|T\|_a=\|T^*\|_a<\infty,\|T^*T\|_a=\|TT^*\|_a<\infty\]
	Thus \(T^*,TT^*,T^*T\in W_{a}^{\infty}(G,A)\). It only needs to prove that \((T^*T)^{-1}\in Inv~ W^\infty(G,A)\).
	
	Obviously, \(T^{*}T\) is a positive operator in \(\mathcal{B} (\ell^{2}(G,A))\). Thus, there must exist constants \(M > 0\), \(N > 0\) such that \(MI\leq T^{*}T\leq NI\), where \(I\) is the identity element of \(\mathcal{B} (\ell^{2}(G,A))\).
	Let
	\begin{equation}
		A = I - \frac{2T^{*}T}{M + N}
	\end{equation}
	Thus the matrix \(A\in \mathcal{B} (\ell^{2}(G,A))\), and
	\begin{equation}\label{2}
		\left \| I - \frac{2T^{*}T}{M + N} \right \|_{2}\leq\frac{N - M}{M + N} < 1\quad \text{and}\quad \left \| I - \frac{2T^{*}T}{M + N} \right \|_{a}<\infty
	\end{equation}
	Also, since
	\((I - A)(I + A + A^2 + A^3 + \cdots)=I\),
	\((I + A + A^2 + A^3 + \cdots)(I - A)=I\)\\
	From (\ref{2}), we get that \((I - A)^{-1}\) exists, and obtain the following estimate
	\begin{align*}
		\left \|( I - A)^{-1} \right \| _{a} &\leq \left \| \sum_{i = 0}^{\infty }A^{i}  \right \|_{a} 
		\leq \sum_{i = 0}^{k}C\|T\|_2^\alpha\cdot\|T\|_2^{\beta}
		<\infty.
	\end{align*}
	Thus \((T^{*}T)^{-1}\in Inv~ W^\infty(G,A) \),
	\(T^{-1}=(T^{*}T)^{-1}T^{*} \in Inv~ W^\infty(G,A)  \), and \(T\) is arbitrarily chosen.
	So it is proved that \(W_{a}^{\infty}(G,A)\) is an inverse-closed subalgebra of \(C^{*}(G,A)\).
\end{proof}

\end{document}